\documentclass[11pt]{amsart}

\usepackage{amsmath}
\usepackage{amssymb}
\usepackage{amsthm}
\usepackage{geometry}
\usepackage{color}
\usepackage{hyperref}
\usepackage{enumerate}
\hypersetup{colorlinks=true}

\newtheorem{thm}{Theorem}[section]
\newtheorem{cor}{Corollary}[section]
\newtheorem{lem}[thm]{Lemma}
\theoremstyle{remark}
\newtheorem{rmk}{Remark}[section]
\newtheorem{que}{Question}[section]

\title[$p$-adic denseness of members of partitions of $\mathbb{N}$ and their ratio sets]{$p$-adic denseness of members of partitions of $\mathbb{N}$ and\\their ratio sets}

\author[P. Miska]{Piotr Miska}
\address{Institute of Mathematics\\ Jagiellonian University in Krak\'{o}w\\ Krak\'{o}w, Poland}
\email{piotrmiska91@gmail.com}

\author[C. Sanna]{Carlo Sanna}
\address{Department of Mathematics\\ Universit\`{a} degli Studi di Torino\\ Torino, Italy}
\email{carlo.sanna.dev@gmail.com}


\begin{document}

\maketitle

\begin{abstract}
The \emph{ratio set} of a set of positive integers $A$ is defined as $R(A) := \{a / b : a, b \in A\}$.
The study of the denseness of $R(A)$ in the set of positive real numbers is a classical topic and, more recently, the denseness in the set of $p$-adic numbers $\mathbb{Q}_p$ has also been investigated.
Let $A_1, \ldots, A_k$ be a partition of $\mathbb{N}$ into $k$ sets.
We prove that for all prime numbers $p$ but at most $\lfloor \log_2 k \rfloor$ exceptions at least one of $R(A_1), \ldots, R(A_k)$ is dense in $\mathbb{Q}_p$.
Moreover, we show that for all prime numbers $p$ but at most $k - 1$ exceptions at least one of $A_1, \ldots, A_k$ is dense in $\mathbb{Z}_p$.
Both these results are optimal in the sense that there exist partitions $A_1, \ldots, A_k$ having exactly $\lfloor \log_2 k \rfloor$, respectively $k - 1$, exceptional prime numbers; and we give explicit constructions for them.
Furthermore, as a corollary, we answer negatively a question raised by Garcia, Hong, \emph{et~al}.
\end{abstract}

\section{Introduction}

The \emph{ratio set} (or \emph{quotient set}) of a set of positive integers $A$ is defined as
\begin{equation*}
R(A) := \{a / b : a,b \in A\} .
\end{equation*}
The study of the denseness of $R(A)$ in the set of positive real numbers $\mathbb{R}_+$ is a classical topic.
For example, Strauch and T{\'o}th~\cite{MR1659159} (see also~\cite{MR1904872}) showed that $R(A)$ is dense in $\mathbb{R}_+$ whenever $A$ has lower asymptotic density at least equal to $1/2$.
Furthermore, Bukor, {\v S}al{\'a}t, and T{\'o}th~\cite{MR1475512} proved that if $\mathbb{N} = A \cup B$ for two disjoint sets $A$ and $B$, then at least one of $R(A)$ or $R(B)$ is dense in $\mathbb{R}_+$. 
On the other hand, Brown, Dairyko, Garcia, Lutz, and Someck~\cite{MR3229105} showed that there exist pairwise disjoint sets $A,B,C \subseteq \mathbb{N}$ such that $\mathbb{N} = A \cup B \cup C$ and none of $R(A)$, $R(B)$, $R(C)$ is dense in $\mathbb{R}_+$.
See also~\cite{MR1635220, MR1390582, MR0242756, MR0248107} for other related results.

More recently, the study of when $R(A)$ is dense in the $p$-adic numbers $\mathbb{Q}_p$, for some prime number $p$, has been initiated.
Garcia and Luca~\cite{MR3593645} proved that the ratio set of the set of Fibonacci numbers is dense in $\mathbb{Q}_p$, for all prime numbers $p$.
Their result has been generalized by Sanna~\cite{MR3668396}, who proved that the ratio set of the $k$-generalized Fibonacci numbers is dense in $\mathbb{Q}_p$, for all integers $k \geq 2$ and prime numbers $p$.
Furthermore, Garcia, Hong, Luca, Pinsker, Sanna, Schechter, and Starr~\cite{MR3670202} gave several results on the denseness of $R(A)$ in $\mathbb{Q}_p$.
In particular, they studied $R(A)$ when $A$ is the set of values of a Lucas sequences, the set of positive integers which are sum of $k$ squares, respectively $k$ cubes, or the union of two geometric progressions.

In this paper, we continued the study of the denseness of $R(A)$ in $\mathbb{Q}_p$.

\section{Denseness of members of partitions of $\mathbb{N}$}

Motivated by the results on partitions of $\mathbb{N}$ mentioned in the introduction, the authors of~\cite{MR3670202} showed that for each prime number $p$ there exists a partition of $\mathbb{N}$ into two sets $A$ and $B$ such that neither $R(A)$ nor $R(B)$ are dense in $\mathbb{Q}_p$~\cite[Example~3.6]{MR3670202}.
Then, they asked the following question~\cite[Problem~3.7]{MR3670202}:

\begin{que}\label{que:RARB}
Is there a partition of $\mathbb{N}$ into two sets $A$ and $B$ such that $R(A)$ and $R(B)$ are dense in no $\mathbb{Q}_p$?\footnote{Actually, in \cite[Problem~3.7]{MR3670202} it is erroneously written ``such that $A$ and $B$ are dense in no $\mathbb{Q}_p$'', so that the answer is obviously: ``Yes, pick any partion into two sets!''. Question~\ref{que:RARB} is the intended question.}
\end{que}

We show that the answer to Question~\ref{que:RARB} is negative.
In fact, we will prove even more. 
Our first result is the following:

\begin{thm}\label{thm:partZp}
Let $A_1, \ldots, A_k$ be a partition of $\mathbb{N}$ into $k$ sets.
Then, for all prime numbers $p$ but at most $k - 1$ exceptions, at least one of $A_1, \ldots, A_k$ is dense in $\mathbb{Z}_p$.
\end{thm}

Then, from Theorem~\ref{thm:partZp} it follows the next corollary, which gives a strong negative answer to Question~\ref{que:RARB}.

\begin{cor}\label{cor:2.1}
Let $A_1, \ldots, A_k$ be a partition of $\mathbb{N}$ into $k$ sets. 
Then, for all prime numbers $p$ but at most $k - 1$ exceptions, at least one of $R(A_1), \ldots, R(A_k)$ is dense in $\mathbb{Q}_p$.
\end{cor}
\begin{proof}
It is easy to prove that if $A_j$ is dense in $\mathbb{Z}_p$ then $R(A_j)$ is dense in $\mathbb{Q}_p$.
Hence, the claim follows from Theorem~\ref{thm:partZp}.
\end{proof}

The proof of Theorem~\ref{thm:partZp} requires just a couple of easy preliminary lemmas.
For positive integers $a$ and $b$, define $a + b\mathbb{N} := \{a + bk : k \in \mathbb{N}\}$.

\begin{lem}\label{lem:partZp1}
Suppose that $(a + b\mathbb{N}) \subseteq A \cup B$ for some positive integers $a,b$ and some disjoint sets $A, B \subseteq \mathbb{N}$.
If $p$ is a prime number such that $p \nmid b$ and $A$ is not dense in $\mathbb{Z}_p$, then there exist positive integers $c$ and $j$ such that $(c + bp^j\mathbb{N}) \subseteq B$.
\end{lem}
\begin{proof}
Since $A$ is not dense in $\mathbb{Z}_p$, there exist positive integers $d,j$ such that $(d + p^j\mathbb{N}) \cap A = \varnothing$. 
Hence, $(a + b\mathbb{N}) \cap (d + p^j\mathbb{N}) \subseteq B$. 
The claim follows by the Chinese Remainder Theorem, which implies that $(a + b\mathbb{N}) \cap (d + p^j\mathbb{N}) = c + bp^j\mathbb{N}$, for some positive integer $c$.
\end{proof}

\begin{lem}\label{lem:partZp2}
Let $a$ and $b$ be positive integers.
Then, $a + b \mathbb{N}$ is dense in $\mathbb{Z}_p$ for all prime numbers $p$ such that $p \nmid b$.
\end{lem}
\begin{proof}
It is follows from the Chinese Remainder Theorem and the fact that $\mathbb{N}$ is dense in~$\mathbb{Z}_p$.
\end{proof}

We are now ready for the proof of Theorem~\ref{thm:partZp}.

\begin{proof}[Proof of Theorem~\ref{thm:partZp}]
For the sake of contradiction, suppose that $p_1, \ldots, p_k$ are $k$ pairwise distinct prime numbers such that none of $A_1, \ldots, A_k$ is dense in $\mathbb{Z}_{p_i}$ for $i = 1,\ldots,k$.
Since $A_1$ is not dense in $\mathbb{Z}_{p_1}$, there exist positive integers $c_1$ and $j_1$ such that $(c_1 + p_1^{j_1}\mathbb{N}) \cap A_1 = \varnothing$.
Hence, $(c_1 + p_1^{j_1}\mathbb{N}) \subseteq A_2 \cup \cdots \cup A_k$ and, thanks to Lemma~\ref{lem:partZp1}, there exist positive integers $c_2$ and $j_2$ such that $(c_2 + p_1^{j_1} p_2^{j_2}\mathbb{N}) \subseteq A_3 \cup \cdots \cup A_k$.
Continuing this process, we get that $(c_{k-1} + p_1^{j_1} \cdots p_{k-1}^{j_{k-1}}\mathbb{N}) \subseteq A_k$, for some positive integers $c_{k-1}, j_1, \ldots, j_{k-1}$.
By Lemma~\ref{lem:partZp2}, this last inclusion implies that $A_k$ is dense in $\mathbb{Z}_{p_k}$, but this contradicts the hypotheses.
\end{proof}

\begin{rmk}\label{rmk:partZp}
In fact, Theorem~\ref{thm:partZp} can be strengthen in the following way: For each partition $A_1, \ldots, A_k$ of $\mathbb{N}$ there exists a member $A_j$ of this partition which is dense in $\mathbb{Z}_p$ for all but at most $k-1$ prime numbers $p$.

Indeed, for the sake of contradiction, suppose that each member $A_j$ of the partition $A_1, \ldots, A_k$ of $\mathbb{N}$ has at least $k$ prime numbers $p$ such that $A_j$ is not dense in $\mathbb{Z}_p$. 
Then we can choose prime numbers $p_1, \ldots, p_k$ such that for each $j\in\{1,\ldots,k\}$ the set $A_j$ is not dense in $\mathbb{Z}_{p_j}$. 
Next, we provide the reasoning from the proof of Theorem~\ref{thm:partZp} to reach a contradiction.
\end{rmk}

The next result shows that the quantity $k - 1$ in Theorem~\ref{thm:partZp} cannot be improved.

\begin{thm}
Let $k \geq 2$ be an integer and let $p_1, \ldots, p_{k-1}$ be pairwise distinct prime numbers.
Then, there exists a partition $A_1, \ldots, A_k$ of $\mathbb{N}$ such that none of $A_1, \ldots, A_k$ is dense in $\mathbb{Z}_{p_i}$ for $i=1, \ldots, k - 1$.
\end{thm}
\begin{proof}
Let $e_1, \ldots, e_{k-1}$ be positive integers such that $p_i^{e_i} \geq k$ for $i=1,\ldots,k-1$, and put 
\begin{equation*}
V := \{0, \ldots, p_1^{e_1} - 1\} \times \cdots \times \{0, \ldots, p_{k-1}^{e_{k-1}} - 1\} .
\end{equation*}
We shall construct a partition $R_0, \ldots, R_{k-1}$ of $V$ (note that the indices of $R_i$ start from $0$) such that if $(r_1, \ldots, r_{k-1}) \in R_j$ then none of the components $r_1, \ldots, r_{k-1}$ is equal to $j$.
Then, we define
\begin{equation*}
A_j := \{n \in \mathbb{N}: \exists (r_1, \ldots, r_{k-1}) \in R_{j-1}, \; \forall i = 1, \ldots, k - 1, \quad n \equiv r_i \pmod {p_i^{e_i}} \} ,
\end{equation*}
for $j = 1, \ldots, k$.
At this point, it follows easily that $A_1, \ldots, A_k$ is a partition of $\mathbb{N}$, and that none of $A_1, \ldots, A_k$ is dense in $\mathbb{Z}_{p_i}$, since $A_{j+1}$ misses the residue class $\equiv j \pmod {p_i^{e_i}}$.

The construction of $R_0, \ldots, R_{k-1}$ is algorithmic.
We start with $R_0, \ldots, R_{k-1}$ all empty.
Then, we pick a vector $\mathbf{x} \in V$ which is not already in $R_0 \cup \cdots \cup R_{k-1}$. 
It is easy to see that there exists some $j\in\{0,...,k-1\}$ such that $j$ does not appear as a component of $\mathbf{x}$. 
We thus throw $\mathbf{x}$ into $R_j$.
We continue this process until all the vectors in $V$ have been picked.

Now, by the construction it is clear that $R_0, \ldots, R_{k-1}$ is a partition of $V$ satisfying the desired property.
\end{proof}

\section{Denseness of ratio sets of members of partitions of $\mathbb{N}$}

The result in Corollary \ref{cor:2.1} is not optimal.
Let $\lfloor x \rfloor$ denote the greatest integer not exceeding $x$, and write $\log_2$ for the base $2$ logarithm.
Our next result is the following:

\begin{thm}\label{thm:partQp}
Let $A_1, \ldots, A_k$ be a partition of $\mathbb{N}$ into $k$ sets.
Then, for all prime numbers $p$ but at most $\lfloor \log_2 k\rfloor$ exceptions, at least one of $R(A_1), \ldots, R(A_k)$ is dense in $\mathbb{Q}_p$.
\end{thm}

Before proving Theorem~\ref{thm:partQp}, we need to introduce some notation.
For a prime number $p$ and a positive integer $w$, we identify the group $(\mathbb{Z}/p^w\mathbb{Z})^*$ with $\{a \in \{1, \ldots, p^w\} : p \nmid a\}$.
Moreover, for each $a \in (\mathbb{Z}/p^w\mathbb{Z})^*$ we define
\begin{equation*}
(a)_{p^w} := \left\{x \in \mathbb{Q}_p^* : x / p^{\nu_p(x)} \equiv a \bmod p^w \right\} ,
\end{equation*}
where, as usual, $\nu_p$ denotes the $p$-adic valuation.
Note that the family of sets
\begin{equation*}
(a)_{p^w} \cap \nu_p^{-1}(s) = \{(a + rp^w)p^s : r \in \mathbb{Z}_p\}
\end{equation*}
where $w$ is a positive integer, $a \in (\mathbb{Z}/p^w\mathbb{Z})^*$, and $s \in \mathbb{Z}$, is a basis of the topology of $\mathbb{Q}_p^*$.
Finally, for all integers $t \leq m$ and for each set $X \subseteq \mathbb{N}$, we define
\begin{equation*}
V_{p^w, t, m} := \left\{(a)_{p^w} \cap \nu_p^{-1}(s) : a \in (\mathbb{Z}/p^w\mathbb{Z})^*,\; s \in \mathbb{Z} \cap {[t, m[} \right\}
\end{equation*}
and
\begin{equation*}
V_{p^w, t, m}(X) := \{I \in V_{p^w, t, m} : X \cap I \neq \varnothing \} .
\end{equation*}
Note that it holds the following trivial upper bound
\begin{equation*}
\#V_{p^w, t, m}(X) \leq \#V_{p^w, t, m} = (m - t) \varphi(p^w) ,
\end{equation*}
where $\varphi$ is the Euler's totient function.

Now we are ready to state a lemma that will be crucial in the proof of Theorem~\ref{thm:partQp}.

\begin{lem}\label{lem:partQp1}
Fix a prime number $p$, two positive integers $w$, $t$, a real number $c > 1/2$, and a set $X \subseteq \mathbb{N}$.
Suppose that $\#V_{p^w, 0, m}(X) \geq c m \,\varphi(p^w)$ for some positive integer $m > t / (2c - 1)$.
Then the ratio set $R(X)$ intersects nontrivially with each set in $V_{p^w, 0, t}$.
\end{lem}
\begin{proof}
Given $(a_0)_{p^w} \cap \nu_p^{-1}(s_0) \in V_{p^w, 0, t}$ we have to prove that $R(X) \cap (a_0)_{p^w} \cap \nu_p^{-1}(s_0) \neq \varnothing$.
For the sake of convenience, define $A := V_{p^w, t, m}(X)$ and
\begin{equation*}
B := \{(a_0a)_{p^w} \cap \nu_p^{-1}(s_0 + s) : (a)_{p^w} \cap \nu_p^{-1}(s) \in V_{p^w, t - s_0, m - s_0}(X) \} .
\end{equation*}
We have
\begin{equation}\label{equ:Abound}
\#A = \#V_{p^w, 0, m}(X) - \#V_{p^w, 0, t}(X) \geq (c m - t) \varphi(p^w) > \frac1{2} (m - t) \varphi(p^w) ,
\end{equation}
where we used the inequality $m > t / (2c - 1)$.
Similarly,
\begin{align}\label{equ:Bbound}
\#B &= \#V_{p^w, 0, m}(X) - \#V_{p^w, 0, t - s_0}(X) - \#V_{p^w, m - s_0, m}(X) \nonumber\\
&\geq (cm - (t-s_0) - s_0)\varphi(p^w) > \frac1{2} (m - t) \varphi(p^w) .
\end{align}
Now $A$ and $B$ are both subsets of $V_{p^w, t, m}$, while $\#V_{p^w, t, m} = (m - t) \varphi(p^w)$.
Therefore, (\ref{equ:Abound}) and~(\ref{equ:Bbound}) imply that $A \cap B \neq \varnothing$.
That is, there exist $(a_1)_{p^w} \cap \nu_p^{-1}(s_1) \in A$ and $(a_2)_{p^w} \cap \nu_p^{-1}(s_2) \in V_{p^w, t - s_0, m - s_0}(X)$ such that $a_1 / a_2 \equiv a_0 \pmod {p^w}$ and $s_1 - s_2 = s_0$, so that $R(X) \cap (a_0)_{p^w} \cap \nu_p^{-1}(s_0) \neq \varnothing$, as claimed.
\end{proof}

\begin{proof}[Proof of Theorem~\ref{thm:partQp}]
For the sake of contradiction, put $\ell := \lfloor \log_2 k \rfloor + 1$ and suppose that $p_1, \ldots, p_\ell$ are $\ell$ pairwise distinct prime numbers such that none of $R(A_1), \ldots, R(A_k)$ is dense in $\mathbb{Q}_{p_i}$ for $i = 1, \ldots, \ell$.
Hence, there exist positive integers $w$ and $t$ such that for each $i \in \{1, \ldots, k\}$ and each $j \in \{1, \ldots, \ell\}$ we have $R(A_i) \cap (a_{i,j})_{p_j^w} \cap \nu_{p_j}^{-1}(s_{i,j}) = \varnothing$, for some $a_{i,j} \in (\mathbb{Z}/p_j^w\mathbb{Z})^*$ and some $s_{i,j} \in \{-(t-1), \ldots, t-1\}$.
Clearly, since ratio sets are closed under taking reciprocals, we can assume $s_{i,j} \geq 0$.
Put $c := 1 / \sqrt[\ell]{k}$, so that $c > 1/2$, and pick a positive integer $m > t / (2c - 1)$.
There are 
\begin{equation*}
N := m^\ell \prod_{j=1}^\ell \varphi(p_j^w)
\end{equation*}
sets of the form
\begin{equation}\label{equ:bigcap}
\bigcap_{j = 1}^\ell \left((a_j)_{p_j^w} \cap \nu_{p_j}^{-1}(s_j)\right)
\end{equation}
where $a_j \in (\mathbb{Z}/p_j^w\mathbb{Z})^*$ and $s_j \in \{0, \ldots, m - 1\}$.
Therefore, there exists $i_0 \in \{1, \ldots, k\}$ such that $A_{i_0}$ intersects nontrivially with at least $N / k$ of the sets of form (\ref{equ:bigcap}).
Consequently, there exists $j_0 \in \{1, \ldots, \ell\}$ such that $A_{i_0}$ intersects nontrivially with at least $c m \varphi(p_{j_0}^w)$ sets of the form $(a)_{p_{j_0}^w} \cap \nu_{p_{j_0}}^{-1}(s)$, where $a \in (\mathbb{Z}/p_{j_0}^w\mathbb{Z})^*$ and $s \in \{0, \ldots, m - 1\}$.
In other words, $\#V_{p_{j_0}^w, 0, m}(A_{i_0}) \geq c m \varphi(p_{j_0}^w)$.
Hence, by Lemma~\ref{lem:partQp1}, the set $R(A_{i_0})$ intersects notrivially with all the sets of the form $(a)_{p_{j_0}^w} \cap \nu_{p_{j_0}}^{-1}(s)$, where $a \in (\mathbb{Z}/p_{j_0}^w\mathbb{Z})^*$ and $s \in \{0, \ldots, t-1\}$, but this is in contradiction with the fact that $R(A_{i_0}) \cap (a_{i_0, j_0})_{p_{j_0}^w} \cap \nu_{p_{j_0}}^{-1}(s_{i_0, j_0}) = \varnothing$.
\end{proof}

The bound $\lfloor \log_2 k \rfloor$ in Theorem~\ref{thm:partQp} is sharp in the following sense:

\begin{thm}\label{thm:optQp}
Let $k \geq 2$ be an integer and let $p_1< \ldots< p_\ell$ be $\ell := \lfloor \log_2 k \rfloor$ pairwise distinct prime numbers.
Then, there exists a partition of $\mathbb{N}$ into $k$ sets $A_1, \ldots, A_k$ such that none of $R(A_1), \ldots, R(A_k)$ is dense in $\mathbb{Q}_{p_i}$ for $i = 1, \ldots, \ell$.
\end{thm}
\begin{proof}
We give two different constructions.
Put $h := 2^\ell$ and let $S_1, \ldots, S_h$ be all the subsets of $\{1, \ldots, \ell\}$.
For $j = 1, \ldots, h$, define
\begin{equation*}
B_j := \{n \in \mathbb{N} : \forall i = 1, \ldots, \ell \quad \nu_{p_i}(n) \equiv \chi_{S_j}(i) \pmod 2\} ,
\end{equation*}
where $\chi_{S_j}$ denotes the characteristic function of $S_j$.
It follows easily that $B_1, \ldots, B_h$ is a partition of $\mathbb{N}$, and that none of $R(B_1), \ldots, R(B_h)$ is dense in $\mathbb{Q}_{p_i}$, for $i=1, \ldots, \ell$, since each $R(B_j)$ contains only rational numbers with even $p_i$-adic valuations.
Finally, since $h \leq k$, the partition $B_1, \ldots, B_h$ can be refined to obtain a partition $A_1, \ldots, A_k$ satisfying the desired property.

The second costruction is similar. 
For $j=1,\ldots,h$, define
\begin{equation*}
C_j=\left\{n\in\mathbb{N} : \left(\frac{n/p_i^{v_{p_i}(n)}}{p_i}\right)=(-1)^{\chi_{S_j}(i)}\mbox{ for each }i\in\{1,...,\ell\}\right\} ,
\end{equation*}
where $\left(\frac{a}{p}\right)$ means the Legendre symbol and in case of $p_1=2$ we put $\left(\frac{a}{2}\right) = a \pmod{4}$.
It follows easily that $C_1, \ldots, C_h$ is a partition of $\mathbb{N}$, and that none of $R(C_1), \ldots, R(C_h)$ is dense in $\mathbb{Q}_{p_i}$, for $i = 1, \ldots, \ell$, since each $R(C_j)$ contains only products of powers of $p_i$ and quadratic residues modulo $p_i$ (in case of $p_1 = 2$ we have only products of powers of $2$ and numbers congruent to $1$ modulo $4$).
Finally, since $h \leq k$, the partition $C_1, \ldots, C_h$ can be refined to obtain a partition $A_1, \ldots, A_k$ satisfying the desired property.
\end{proof}

In the light of Remark~\ref{rmk:partZp} it is worth to ask a the following question.

\begin{que}
Let us fix a positive integer $k$. 
What then is the least number $m=m(k)$ such that for each partition $A_1, \ldots, A_k$ of $\mathbb{N}$ there exists a member $A_j$ of this partition such that $R(A_j)$ is dense in $\mathbb{Q}_p$ for all but at most $m$ prime numbers $p$?
\end{que}

In virtue of Remark~\ref{rmk:partZp} we know that $m(k)$ exists and $m(k)\leq k-1$. 
On the other hand, by Theorem~\ref{thm:optQp} the value $m(k)$ is not less than $\left\lfloor\log_2 k\right\rfloor$.

\bibliographystyle{plain}

\end{document}